\documentclass{amsart}

\usepackage[utf8]{inputenc}
\usepackage{hyperref}
\usepackage{amssymb}
\usepackage{amsmath}
\usepackage{amsthm}

\newtheorem{theorem}{Theorem}[section]
\newtheorem{lemma}{Lemma}[section]
 \theoremstyle{remark}
\newtheorem{remark}{Remark}[section]

\newcommand{\N}{\mathbb{N}}
\newcommand{\R}{\mathbb{R}}
\newcommand{\Z}{\mathbb{Z}}
\newcommand{\be}{\begin{equation}}
\newcommand{\ee}{\end{equation}}

\begin{document}

\title[Embedding Theorems for Potential Spaces of Radial Functions]{Elementary proofs of  Embedding  Theorems for Potential Spaces of Radial Functions}

\author{Pablo L.  De N\'apoli}
\address{IMAS (UBA-CONICET) and Departamento de Matem\'atica\\
Facultad de Ciencias Exactas y Naturales\\
Universidad de Buenos Aires\\
Ciudad Universitaria\\
1428 Buenos Aires\\
Argentina}
\email{pdenapo@dm.uba.ar}

\author{Irene Drelichman}
\address{IMAS (UBA-CONICET) and Departamento de Matem\'atica\\
Facultad de Ciencias Exactas y Naturales\\
Universidad de Buenos Aires\\
Ciudad Universitaria\\
1428 Buenos Aires\\
Argentina}
\email{irene@drelichman.com}

\thanks{Supported by ANPCyT under grant PICT 1675/2010, by CONICET under  grant PIP 1420090100230 and by Universidad de Buenos Aires under grant 20020090100067. The authors are members of
CONICET, Argentina.}



\date{April 29, 2014}

\begin{abstract}
We present elementary proofs of weighted embedding theorems for radial potential spaces and some generalizations of Ni's and Strauss' inequalities in this setting. 
\end{abstract}

\maketitle

\section{Introduction: Sobolev spaces and embedding theorems}

The aim of this note is twofold: to review some known results from the theory of radial functions in Sobolev (potential) spaces, and to extend some of this results to the setting of weighted radial spaces. In both cases, we provide new elementary proofs that avoid the use of interpolation theory and sophisticated tools such as atomic or wavelet decompositions. 

This will have, at times, the limitation of not giving the most general possible result, in which case we will do our best to provide suitable references. However, we believe that the proofs presented here have the merit of being closer in spirit to some classical theorems (such as the Sobolev embedding theorem) and that the results obtained are good enough for common applications to the theory of partial differential equations.

Before we state the results we are interested in, let us  quickly recall that for Sobolev spaces of integer order, the most classical definition is in terms of derivatives: given a domain $\Omega \subset \R^n$, for integer $k \in \N$,
$$W^{k,p}(\Omega) = \left \{ u \in L^p(\Omega) : D^\alpha u
\in L^p(\Omega) \; \hbox{for any} \; \alpha \in \N_0^n  \; 
\hbox{with} \; |\alpha|\leq k \right \} $$
where the derivatives $D^\alpha u$ are understood in weak (or
distributional) sense. It is also usual to denote $W^{k,2}(\Omega)$ by
$H^k(\Omega)$.

When it comes to Sobolev spaces of fractional order, there exist several definitions in the
literature, among others:
\begin{enumerate}

\item[a)] Sobolev spaces of fractional order based on $L^2(\R^n)$ can be defined in terms of the Fourier transform:
for real $s \geq 0$, let
$$ H^{s}(\R^n) = 
\{ u \in L^2(\R^n) : \int_{\R^n} 
|\widehat{u}(\omega)|^2 (1+|\omega|^2)^{s} \; d\omega < \infty   \} $$
For integer $s$, we have that $H^s(\R^n)= W^{k,2}(\R^n)$.

\item[b)] One way to define spaces of fractional order based on $L^p(\R^n)$
when $p\neq 2$ is given by the classical potential spaces $H^{s,p}(\R^n)$ defined by:
$$ H^{s,p}(\R^n) = \{ u : u = (I-\Delta)^{-s/2} f \; \hbox{with} \; 
f\in L^p(\R^n) \} $$
Here, $1<p<\infty$ and the fractional power $(I-\Delta)^{-s/2}$ can be defined by means of the Fourier transform (for functions in the Schwarz class):
$$ (I-\Delta)^{-s/2} f = \mathcal{F}^{-1} ( (1+|\omega|^2)^{-s/2}
\mathcal{F}(f) ) = G_s * f $$
where
\begin{equation}
\label{Bessel-potential}
G_s(\omega) = \mathcal{F}^{-1} ( (1+|\omega|^2)^{-s/2} ) =\frac{(2 \sqrt{\pi})^{-n}}{\Gamma (s/2)} 
\int_0^\infty e^{-t} e^{\frac{-|x|^2}{4t}} t^{\frac{s-n}{2}} \; \frac{dt}{t} 
\end{equation}
is called the Bessel potential. Some classical references on  potential spaces are 
\cite{Stein}, \cite{AS} and \cite{Calderon}.  Notice that this family includes all the previous spaces (when $\Omega=\R^n$) since we have that
$ H^{1,p}(\R^n) = W^{1,p}(\R^n)$ for $1 < p < \infty $ by \cite[Theorem 7]{Calderon}.
Also $H^{s,2}(\R^n) = H^{s}(\R^n)$ for any $s\geq 0$ by Plancherel's
theorem. 

\item[c)] Another (non-equivalent) definition of fractional Sobolev spaces
is given by the Aronszajn-Gagliardo-Slobodeckij spaces, for $0<s<1, 1\leq p < \infty$:
$$ W^{s,p}(\Omega)= \left\{ u \in L^p(\Omega) :
\frac{|u(x)-u(y)|}{|x-y|^{\frac{n}{p}+s}} \in L^p(\Omega \times \Omega)
\right \}  $$
See \cite{DPV} for a full exposition. We have that $H^{s,2}(\R^n)=
W^{s,2}(\R^n)=H^s(\R^n)$ for any $0<s<1$, but for $p\neq 2$, $H^{s,p}(\R^n)$ and
$W^{s,p}(\R^n)$ are different.

\item[d)] Even more general families of functional spaces are those of Besov-Lipschitz and 
Triebel-Lizorkin spaces. Indeed, the Littlewood-Paley theory says that
$H^{s,p}(\R^n)$ coincides with the Triebel-Lizorkin space $F^{s}_{p,2}(\R^n)$,
whereas $W^{s,p}(\R^n)$ (as defined in c))  coincides with the
Besov-Lipschitz space $B^s_{p,2}(\R^n)$.
\end{enumerate}

See for instance \cite{Adams} for a full exposition of the theory of Sobolev
spaces, and \cite{Triebel} for a discussion of the relations of the different
functional spaces with the scales of Besov-Lipschitz and Triebel-Lizorkin
spaces.

In this note, we shall focus on potential spaces.  A key result in the theory  is the Sobolev embedding
theorem, that reads as follows:

\begin{theorem}[classical Sobolev embedding]\cite[section 10]{AS},\cite[Theorem 6] {Calderon}
Assume that $sp<n$ and define the critical Sobolev  
exponent $p^*$ as
$$ p^* := \frac{np}{n-sp} $$
Then, there is a continuous embedding 
\be H^{s,p}(\R^n) \subset L^{q}(\R^n) \quad 
\hbox{for } \; p \leq q \leq p^*
\label{Sobolev-embedding}
\ee 
\label{Sobolev-embedding-theorem}
\end{theorem}

Embedding theorems like this one play a central role when one wants to apply the direct
methods of  calculus of variations to obtain solutions of nonlinear
partial differential equations. Consider, for simplicity, the following elliptic
nonlinear boundary-value model problem in a smooth domain $\Omega \subset \R^n$: 
\be
\label{problema-modelo}
\left\{
\begin{array}{rclll}
-\Delta u & = & u^{q-1} & \hbox{in} & \Omega \\
u & = &0 & \hbox{on}\; & \partial \Omega \\
\end{array}
\right.
\ee
It is well known that solutions of this problem can be obtained as critical points of the
energy functional 
$$ J(u)= \frac{1}{2} \int_{\Omega} |\nabla u|^2 - 
\frac{1}{q} \int_{\Omega} |u|^q $$ 
in the natural energy space, which is the 
Sobolev space $H^1_0(\Omega)$, i.e., the closure of $C_0(\Omega)$
in $H^1(\Omega)$  (observe that we need to restrict $u$ to this subspace in order to
reflect the Dirichlet boundary condition; this implies in turn that the
functional is coercive). The Sobolev embedding gives that
$$ H^1(\Omega) \subset L^q(\Omega) \; \quad 2 \leq q \leq
2^*=\frac{2n}{n-2} $$
which implies that $J$ is well defined (and actually of class $C^1$) 
for $q$ within that range. 

Then, one can use minimax theorems like the mountain pass theorem of  
Ambrosetti and Rabinowitz (see  \cite{Rabinowitz} or \cite{Willem}), in order to obtain critical points of $J$ that
correspond to non-trivial solutions of our elliptic problem. However, this
kind of theorems usually need some compactness assumption like the
Palais-Smale condition:

\begin{quote}
For any sequence $u_n$ in the space $H^1_0(\Omega)$ 
such that $J(u_n) \to 0$ and $J^\prime(u_n) \to 0$, there
exist a convergent subsequence $u_{n_k}$
\end{quote}

In the subcritical case $2 < q < 2^*$, this condition is easily verified by
using the  Rellich-Kondrachov theorem: 

\begin{theorem}
If $\Omega \subset \R^n$ is a bounded domain, and $2 \leq q < 2^*$ then 
the Sobolev embedding
$$ H^1_0(\Omega) \subset L^q(\Omega) $$
is \emph{compact}.
\end{theorem}

See \cite{Rabinowitz} or \cite{Willem} for more information on variational methods for
nonlinear elliptic problems.

When $\Omega$ is unbounded, for instance $\Omega=\R^n$, the situation is
different, since the embedding \eqref{Sobolev-embedding} is continuous but not compact, due to the invariance of the 
$H^{s,p}$ and $L^q$ norms under
translations. Indeed, if we choose $u \in C_0^\infty(\Omega)$ the sequence
$$ u_n(x)= u(x+n v) $$
where $v\neq 0$ is a fixed vector in $\R^n$, is a bounded sequence in
$H^{s,p}(\R^n)$ without any convergent subsequence.

Assume now that we want to find solutions of the analogous subcritical problem
\be -\Delta u + u = u^{q-1} \quad 2 < q < 2^*=\frac{2n}{n-1} 
\label{elliptic-rn} \ee
in $\R^n$. Due to the lack of compactness of the Sobolev embedding, the
previous approach does not work. However, we can still get compactness and
hence non-trivial weak solutions, by restricting the energy functional  
$$ J(u)= \frac{1}{2} \int_{\R^n} |\nabla u|^2 + 
 \frac{1}{2} \int_{\R^n} |u|^2 
-\frac{1}{q} \int_{\R^n} |u|^q $$ 
to the subspace $H^1_{rad}(\R^n)$ of functions in $H^1(\R^n)$ with radial 
symmetry,
i.e. such that $u(x)=u_0(|x|)$. Indeed we can make use of the following result, due to
P. L. Lions \cite{Lions}

\begin{theorem}
Let $n \geq 2$. For  $2 < q < 2^*$, the embedding
$$ H^1_{rad}(\R^n) \subset L^q(\R^n) $$
is compact.
\end{theorem}

\begin{remark}
When $q=2^*$ we don't have compactness even for radial functions, due to the
rescaling invariance. The same happens for $q=2$, since otherwise the
Laplacian would have a strictly positive principal eigenvalue in $\R^n$. 
\end{remark}

Moreover, the orthogonal group $O(n)$ acts naturally on $H^1(\R^n)$ by
$$ (g \cdot u)(x)= u(g \cdot x) \quad g \in O(n) $$
and the functional $J$ is invariant by this action. Futhermore, $H^1_{rad}$
is precisely the subspace of invariant functions under this action. Hence,
the principle of symmetric criticality \cite{Palais} implies that the critical points of
$J$ on $H^1_{rad}(\R^n)$ are also critical points of $J$ on  $H^1(\R^n)$, and
hence weak solutions of the elliptic problem \eqref{elliptic-rn}. 

This phenomenon of having better embeddings for spaces of radially symmetric 
functions (or more generally, for functions invariant under some subgroup of
the orthogonal group) is well known, and goes back
to the pioneering work by W. M. Ni \cite{Ni} and W. Strauss \cite{Strauss},
who proved the following theorems:

\begin{theorem}[\cite{Ni}]
Let $B$ be the unit ball in $\R^n \; (n\geq 3)$,
Then every function in $H^1_{0,rad}(B)$ is almost everywhere equal to a function
in $C(\overline{B}-\{0\})$ such that
$$ |u(x)| \leq C |x|^{-(n-2)/2} \| \nabla u \|_{L^2}  $$ 
with $C=(\omega_n (n-2))^{-1/2}$,  
$\omega_n$ being the surface area of $\partial B$.
\label{Ni-H1}
\end{theorem}

\begin{theorem}[\cite{Strauss}]
Let $n \geq 2$. Then every function in $H^1_{rad}(\R^n)$ is almost everywhere equal to a function
in $C(\R^{n}-\{0\})$ such that
$$ |u(x)| \leq C \; |x|^{-(n-1)/2} \; \| u \|_{H^1} $$
\end{theorem}

In both cases, by density, it is enough to establish these inequalities for
smooth radial functions. Once  the inequality is known for smooth functions, if $u_n$ is a sequence of smooth functions convergent to $u$ in $H^1$, the inequality implies that $u_n$ converges uniformly on compact sets of $\R^n-\{0\}$ and hence, $u$ is equal almost everywhere to a continuous function outside the origin that satisfies the same inequality. In the sequel we shall use this observation without further comments.

\medskip 

Even though both inequalities look quite similar at first glance, they have
different features:
\begin{enumerate}
\item Observe that the exponent of $|x|$ in each inequality is different.
As a consequence, for a function in $H^1_{rad}(\R^n)$ the estimate provided by Ni's inequality is sharper near the
origin, whereas Strauss' inequality is better at infinity.
\item Ni's inequality is invariant under scaling. As a consequence, one has a.e.
$$ |u(x)| \leq C |x|^{-(n-2)/2} \| \nabla u \|_{L^2}$$
for $u \in H^1_{rad}(\R^n)$. On the other hand, Strauss' inequality is
not invariant by rescaling. However,  replacing $u$ by $u(\lambda x)$ and
minimizing over $\lambda$, one can get the following rescaling-invariant
version:
$$ |u(x)| \leq C \; |x|^{-(n-1)/2} \;  \| u \|_{L^2}^{1/2} \;  \| \nabla  u 
\|_{L^2}^{1/2}  $$
\end{enumerate}

For elliptic problems other than \eqref{problema-modelo} that  involve other operators,  we need different functional spaces, for instance,  for the case of the $p$-Laplacian $\Delta_p u := \hbox{div} (|\nabla u|^{p-2} \nabla u) $ the natural space is $W^{1,p}_0(\Omega)$ (see, e.g., \cite{DJM}), and for the fractional Laplacian $(-\Delta)^{-s/2}$ the natural space is $H^s(\R^n)$ (see, e.g., \cite{Secchi}). The advantage of working with potential spaces is, as noted above,  that they include all these cases in a unified framework. As  in the case of $H^1$ and $H^1_{rad}$, the corresponding subspaces of radially symmetric fuctions of the above spaces will be denoted by the subscript \emph{rad}.

This paper is organized as follows. In Section 2 and Section 3 we recall some results on fractional integrals and derivatives, respectively. Then we give two easy proofs of  Strauss' inequality for  potential spaces, one for $p=2$ (in Section 4) and another one for general $p$ (in Section 5). Section 6 is devoted to embedding theorems with power weights for radial functions, which include a generalization of Ni's inequalty in $\R^n$. In Section 7 we analyze the compactness of the embeddings, both in the unweighted case (giving an alternative proof of Lions' theorem that avoids the use of complex interpolation) and in the weighted case. Finally, in Section 8 we discuss a generalization of Ni's inequality and embedding theorems for potential spaces in a ball.

\section{Fractional Integral estimates and embedding theorems}

Several ways of proving the Sobolev embedding are known. One of the most
classical (which goes back to the original work by Sobolev), involves the operator 
$$ I^s f(x) = c(n,s) \int_{\R^{n}} \frac{f(y)}{|x-y|^{n-s}} \; dy \quad  
0< s < n $$
(whre $c(n,s)$ is a normalization constant), which is known as fractional integral or Riesz potential.  This operator provides an integral representation of the negative fractional
powers of the Laplacian, i.e:
$$ I^s(f)= (-\Delta)^{-s/2} f $$
for smooth functions $f \in \mathcal{S}(\mathbb{R}^n)$. See also \cite{MSP} for an extension to a space of distributions. 

A basic result on this operator is the following theorem concerning its
behaviour in the classical Lebesgue spaces:

\begin{theorem}[Hardy-Littlewood-Sobolev]
If $p>1$ and $\frac{1}{q}=\frac{1}{p}-\frac{s}{n}$ then
$I^s$ is of strong type $(p,q)$, i.e: there exist $C>0$ such that 
$$  \| I^s f \|_{L^q(\R^n)} \leq C \; \| f \|_{L^p(\R^n)} $$
\label{HLS-theorem}
\end{theorem}

We recall that, using this result,  the Sobolev embedding (Theorem 
\ref{Sobolev-embedding}) follows easily. Indeed, it is well know that the Bessel potential
satisfies the the following
asymptotics for $s<n$ (see,  e.g., \cite[Theorem 4]{Calderon}),
\be G_s(x) \simeq C \left\{
\begin{array}{rcl}
|x|^{s-n} & \hbox{if} & |x|\leq 2 \\
e^{-|x|/2} & \hbox{if} & |x|\geq 2 \\
\end{array}
\right. \label{asymptotics} \ee

Since, in particular, $G_s \in L^1(\R^n)$, one immediately gets the  embedding
\be H^{s,p}(\R^n) \subset L^p(\R^n) \; \label{embedding-Lp} \ee

On the other hand, \eqref{asymptotics} gives the pointwise bound 
\be |G_s*f(x)| \leq C \; I^s(|f|)(x) \label{fractional-integral-bound} \ee
whence, using the Hardy-Littlewood-Sobolev  theorem, \eqref{embedding-Lp} 
and H\"older's inequality one obtains the Sobolev embedding \eqref{Sobolev-embedding}.

It is therefore natural to ask ourselves which is the corresponding extension of Theorem \ref{HLS-theorem} in the weighted case. In this article, we
are mainly interested in power weights, for which the the following 
theorem was proved in \cite{SW} by E. Stein and G. Weiss:

\begin{theorem}\cite[Theorem B*]{SW}
\label{stein-weiss}
Let $n\ge 1$, $0<s<n, 1<p<\infty, \alpha<\frac{n}{p'},
\beta<\frac{n}{q}, \alpha + \beta \ge 0$, and
$\frac{1}{q}=\frac{1}{p}+\frac{\alpha+\beta-s}{n}$. If $p\le
q<\infty$, then the inequality
$$
\||x|^{-\beta} I^s f\|_{L^q(\mathbb{R}^n)} \le C \||x|^\alpha
f \|_{L^p(\mathbb{R}^n)}
$$
holds for any $f\in L^p(\mathbb{R}^n, |x|^{p\alpha} dx)$, where $C$ is
independent of $f$.
\end{theorem}

For our purposes, it will be very important to notice that when we restrict 
our attention to functions with radial symmetry, we can get
a better result. Indeed, in \cite{DDD2} the authors proved:

\begin{theorem}
\label{theorem-DDD2}
Let $n\ge 1$, $0<s<n, 1<p<\infty, \alpha<\frac{n}{p'},
\beta<\frac{n}{q}, \alpha + \beta \ge (n-1)(\frac{1}{q}-\frac{1}{p})$, and
$\frac{1}{q}=\frac{1}{p}+\frac{\alpha+\beta-s}{n}$.  If $p\le
q<\infty$, then the inequality
$$
\||x|^{-\beta} I^s f \|_{L^q(\mathbb{R}^n)} \le C \||x|^\alpha
f \|_{L^p(\mathbb{R}^n)}
$$
holds for all radially symmetric $f \in L^p(\mathbb{R}^n, |x|^{p\alpha} dx)$,
where $C$ is independent of $f$.
Moreover, the result also holds for $p=1$ provided $\alpha+\beta > (n-1)(\frac{1}{q}-\frac{1}{p}).$
\end{theorem}

Different proofs of this result were also given by B. Rubin in \cite{Rubin} (except for $p=1$) and by J. Duoandikoetxea in \cite{D}. A more general theorem involving inequalities with angular integrability
was proved by P. D'Ancona and R. Luc\`a in \cite{DL}. 

\begin{remark}
As observed in \cite{DL}, Theorem \ref{theorem-DDD2} also holds for $q=\infty$ (with essentially the same proof) provided that $\alpha+\beta>(n-1)(\frac{1}{p}-\frac{1}{q})$.
\label{remark-infty}
\end{remark}
\section{Riesz Fractional Derivatives}

\label{section-Riesz-derivatives}

It is well known that the classical potential spaces can be characterized in terms of certain hypersingular integrals, which provide a left inverse for
the Riesz potentials. We consider for that purpose, the hypersingular integral 
\be D^s u(x):=\frac{1}{d_{n,l}(s)} \int_{\R^n} \frac{(\Delta^l_y u)(x)}{|y|^{n+s}} \; dy, \quad l>s>0 \label{hypersingular-integral} \ee
where 
$$ (\Delta^l_y u)(x) = \sum_{j=0}^l (-1)^j \binom{l}{j} u(x+(l-j)y) $$  
denotes the finite difference of order $l \in \N$ of the function $u$, and $d_{n,l}(\alpha)$ is a certain normalization constant, which is chosen so that the construction 
\eqref{hypersingular-integral} does not depend on $l$ (see \cite[Chapter 3]{Samko}
for details). We denote by
$$ D_\varepsilon^s u(x) =  \frac{1}{d_{n,l}(s)} \int_{|x|>\varepsilon} 
\frac{(\Delta^l_y u)(x)}{|y|^{n+s}} \; dy $$

Then, we have the following result:
\begin{theorem}\cite[Theorem 26.3]{SKM}
If $u\in L^p(\R^n)$, and $1\le p<n/s$ then
$$ \lim_{\varepsilon \to 0} D_\varepsilon I^s u(x) = u(x) $$
in $L^p(\R^n)$.
\end{theorem}

Another way of inverting the Riesz potential (which holds for every $s$) is 
given by the Marchaud method (see \cite{Rubin-book}).

Then, we have the following characterization theorem due to E. M. Stein (for
$0<s<2$, \cite{Stein2}) and S. Samko  \cite[Theorem 27.3]{SKM} (see also 
\cite[Proposition 2.1]{Kurokawa}).  

\begin{theorem}
Let $1<p<\infty$, $0<s<\infty$, $l>[(s+1)/2]$ if $s$ is not an odd integer, and $l=s$ if 
$s$ is an odd integer. Then $u \in H^{s,p}(\R^n)$ if and only if
$u \in L^p(\R^n)$ and the limit
$$ D^s u = \lim_{\varepsilon \to 0} D^s_\varepsilon u $$
exists in the norm of $L^p(\R^n)$. Moreover, 
$$ \| u \|_{H^{s,p}} \simeq \| u \|_{L^p} + \| D^s u \|_{L^p}  $$
\label{characterization}
\end{theorem}

\begin{remark}
For $0<s<2$, $D^s$ coincides with the fractional Laplacian (as defined, e.g., in
 \cite{DPV}) for smooth functions, i.e. we have that
$$ D^s u= (-\Delta)^{s/2} u $$  
\end{remark}

\begin{remark}
For $p=2$, we can express the $L^2$-norm of $D^s u$ in terms of the Fourier
transform (as a consequence of Plancherel's theorem, see, e.g., \cite[Proposition 3.4]{DPV})
$$ \| D^s u \|_{L^2} \simeq \int_{\R^n} |\widehat{u}(u)|^2 |\omega|^{2s} \; d\omega $$
\end{remark}

Other related characterizations of  potential spaces are given in \cite{Strichartz} and
\cite{Dorronsoro}.

\section{A version of Strauss' Lemma for $p=2$ using the Fourier transform}

In this section, we give an elementary proof of  Strauss' inequality for 
$p=2$, namely, 
\begin{theorem} 
Let $u \in H^s_{rad}(\R^n)$ and $s>1/2$. Then $u$ is almost everywhere equal to a continuous function in $\R^n-\{0\}$ that satisfies
$$ |u(x)| \leq C |x|^{-(n-1)/2} \; \| u \|_{L^2}^{1-\frac{1}{2s}} \; \| D^s u \|_{L^2}^{\frac{1}{2s}} $$
\end{theorem}

\begin{remark}
For an application of this result  to non-linear equations involving the fractional Laplacian operator see \cite{Secchi}. 
\end{remark}

We will make use of the following well-known lemmas (see, e.g., \cite[Chapter 3]{SWbook}): 

\begin{lemma}[Fourier transform of radial functions]
Let $u \in L^2_{rad}(\R^n)$ be a radial function, $u(x)=u_{rad}(|x|)$.
Then its Fourier transform $\widehat{u}$ is also radial, and it is given by:
$$ \widehat{u}(\omega) = (2\pi)^{n/2} |\omega|^{-\nu} \int_0^\infty u_{rad}(r)
\; J_{\nu}(r|x|) \; r^{n/2}\; dr $$
where $\nu=\frac{n}{2} -1$ and $J_\nu$ denotes the Bessel function of
order $\nu$.
\label{radial-Fourier}
\end{lemma}

\begin{lemma}[Asymptotics of Bessel functions]
If $\lambda>-1/2$, then there exists a constant $C=C(\lambda)$ such that
$$ |J_\lambda(r)| \leq C \; r^{-1/2} \; \forall \; r \geq 0 $$
\label{Bessel}
\end{lemma}

We observe that $u(x)$ can be reconstructed from $\widehat{u}$ using the Fourier
inversion formula, and that for radial functions (that are even) the
Fourier transform and the inverse Fourier transform coincide up to a
constant factor. Hence, from Lemma \ref{radial-Fourier}, we have that:

$$ |u(x)| \leq C |\omega|^{-n/2+1} \int_0^\infty |(\widehat{u})_{rad}(r)| 
\; |J_\nu(r|\omega)| \; r^{n/2} \; dr $$

Moreover, using Lemma \ref{Bessel}, we have that:

$$ |u(x)| \leq C \; |\omega|^{-(n-1)/2} \int_0^\infty
|(\widehat{u})_{rad}(r)| \; r^{(n-1)/2} \; dr $$
 
In order to bound this expression, we split the integral into parts: the low frequency part from
$0$ to $r_0$, and the high frequency part from $r_0$ to $\infty$, where $r_0$ will be chosen
later, and we bound each part using the Cauchy-Schwarz inequality.

With respect to the low frequency part, we have that:
$$ \int_0^{r_0} |(\widehat{u})_{rad}(r)| 
\; r^{(n-1)/2} \; dr \leq \left(\int_0^{r_0} |\widehat{u}(x)|^2 \; r^{n-1} \; dr
\right)^{1/2} 
\left( \int_0^{r_0} dr \right)^{1/2} \leq 
\| u \|_{L^2} \; r_0^{1/2} $$

With respect to the high frequency part, we have that:

$$ \int_{r_0}^\infty |(\widehat{u})_{rad}(r)| 
\; r^{(n-1)/2} \; dr 
=  \int_{r_0}^\infty |(\widehat{u})_{rad}(r)| 
\; r^{(n-1)/2} \; r^{s} \; r^{-s} \; dr $$
$$ \leq \left( \int_{r_0}^\infty |(\widehat{u})_{rad}(r)|^{2} r^{n-1} r^{2s}
\;dr  \right)^{1/2}
\left( \int_{r_0}^\infty r^{-2s} \; dr \right)^{1/2} 
\leq \| D^su \|_{L^2} \; 
\left( \frac{r_0^{-2s-1}}{2s-1} \right)^{1/2} $$

Summing up, we obtain the following estimate:

$$ |u(x)| \leq C \; |x|^{-(n-1)/2} \left( 
r_0^{1/2} \; \| u \|_{L^2}  + 
\left( \frac{r_0^{-2s-1}}{2s-1} \right)^{1/2}
\; \| D^su \|_{L^2} \right) $$

which implies

\begin{equation}
 |u(x)|^2 \leq C^2 \; |x|^{-(n-1)} \left( 
r_0 \; \| u \|_{L^2}^2  + 
\left( \frac{r_0^{-2s-1}}{2s-1} \right)
\; \| D^su \|_{L^2}^2 \right) 
\label{inequallity-to-minimize}
\end{equation}

Now we apply the following straightforward lemma:

\begin{lemma}
Consider the function of the real variable $f(t)= \alpha t + \beta t^{-\gamma}$ with
$\gamma>0$. Then the function $f$ archives its minimum at the point 
$$ t_0= \left( \frac{\beta \gamma}{\alpha} \right)^{1/(\gamma+1)} $$
at which
$$ f(t_0)= C(\gamma) \alpha^{\gamma/(\gamma+1)} \beta^{1/(\gamma+1)} $$
\end{lemma}

So,  taking $\gamma=2s-1$, $\alpha= \; \| u \|_{L^2}^2 $, and 
$\beta= {(2s-1)}^{(2s-1)} \; \| D^su \|_{L^2}^2$ we get that the minimum of
the right hand side of \eqref{inequallity-to-minimize} is achieved at a point 
$r_0$, which gives the required inequality 

$$ |u(x)| \leq C |x|^{-(n-1)/2} \; \| u \|_{L^2}^{1-\frac{1}{2s}} \; 
\| D^s u \|_{L^2}^{\frac{1}{2s}} $$

since

$$ \frac{1}{\gamma+1} = \frac{1}{2s} \quad \frac{\gamma}{\gamma+1}
= \frac{2s-1}{2s} = 1 - \frac{1}{2s} $$

\begin{remark}
A similar technique was used in \cite{Cho-Ozawa} to derive a whole family of related
 inequalities, including a generalization of Ni's inequality. 
\end{remark}

\begin{remark}
We observe that Lemma \ref{radial-Fourier} actually expresses the Fourier
transform of a radial function in terms of the (modified) Hankel transform. As a
consequence, many results on embeddings for spaces of radial functions can
also be derived by using results for the Hankel transform. For instance in 
\cite{DDD1}, we have used this technique to prove Theorem
\ref{embedding-theorem} for $p=2$, using results of  L. De Carli \cite{DeCarli}. 
Also some of the results of this paper can be
alternatively derived from  results obtained by A. Nowak and K. Stempak in \cite{Nowak-Stempak}.
\end{remark}

\section{Proof of Strauss' inequality }

In this section we prove a version of Strauss' inequality for potential spaces for any $p$.  We begin by  the following lemma: 

\begin{lemma}
Let $f \in L^p(\R^n)$ be a radial function, then
$$ |f*\chi_{B(0,R)}(x) | \le C R^{n-1/p } |x|^{-(n-1)/p} \Vert f \Vert_{L^p(\R^n)} $$
\label{radial-conv}
\end{lemma}

\begin{proof}
Observe first that, for $y\in B(x,R)$ we have that $|x|-R\le |y| \le R+|x|$. Now, taking polar coordinates
\begin{align*}
y&=ry^\prime \quad r=|y| \quad y^\prime \in S^{n-1}\\
x&=\rho x^\prime \quad \rho=|x| \quad x^\prime \in S^{n-1}
\end{align*}
we have 
$$ f*\chi_{B(0,R)}(x)=  \int_{\rho - R}^{\rho + R}  f_0(r) \left( \int_{S^{n-1}} \chi_{B(x,R)}(ry^\prime)  \; dy^\prime \right) r^{n-1}\; dr 
$$
To bound the inner integral, observe that  $ \chi_{B(x,R)}(ry^\prime)= \chi_{[t_0,1]}(x^\prime \cdot y^\prime)$ with $t_0=\frac{r^2+\rho^2-R^2}{2r\rho}$, and that $t_0\le 1$ because $\rho-R\le r\le R+\rho$. 

Consider first the case $\rho> 2R$. It follows that $t_0> -1$ and integrating over the sphere
(see  \cite[Lemma 4.1]{DDD2} for details),
$$\int_{S^{n-1}} \chi_{B(x,R)}(ry^\prime)  \; dy^\prime =\int_{S^{n-1}} \chi_{[t_0,1]}(x^\prime \cdot y^\prime)  \; dy^\prime=
\omega_{n-2} \int_{-1}^1 \chi_{[t_0,1]}(t) \left( 1-t^2 \right)^{\frac{n-3}{2}} \; dt. $$
where $\omega_{n-2}$ denotes the area of $S^{n-2}$. Therefore, 
\begin{align*} 
|f*\chi_{B(0,R)}(x)| &= \omega_{n-2} \left|\int_{\rho - R}^{\rho + R}    \int_{t_0}^1 f_0(r)  \left( 1-t^2 \right)^{\frac{n-3}{2}} r^{n-1} \; dt \; dr \right| \\
& \le  \omega_{n-2} \; \left( \int_{\rho - R}^{\rho + R}  |f_0(r)|^p r^{n-1}  \; dr \right)^{1/p} \left( \int_{\rho - R}^{\rho + R}   \left( \int_{t_0}^1   \left( 1-t^2 \right)^{\frac{n-3}{2}}\, dt  \right)^{p^\prime}   r^{n-1} \, dr \right)^{1/p^\prime } \\
& = \omega_{n-2} \;  \|f\|_{L^p(\R^n)} \, A(\rho)
\end{align*}

Notice that 
$$\int_{t_0}^1   \left( 1-t^2 \right)^{\frac{n-3}{2}} \; dt =  \int_{t_0}^1   \left( 1-t \right)^{\frac{n-3}{2}} \left( 1+t \right)^{\frac{n-3}{2}}\; dt \le  C \left( 1 - t_0 \right)^{\frac{n-1}{2}}$$
which implies
\begin{align*} 
A ( \rho ) \le & C \left( \int_{\rho - R}^{\rho + R} \left( 1 - \frac{r^2+\rho^2-R^2}{2r \rho } \right)^{\frac{(n-1)p^\prime}{2}} r^{n-1}  \; dr \right)^{1/p^\prime } \\
\le & C \left( \int_{\rho - R}^{\rho + R}   \left( 1 -  \frac{1+(r / \rho )^2-(R / \rho )^2}{2(r / \rho )} \right)^{\frac{(n-1)p^\prime }{2}} r^{n-1}  \; dr \right)^{1/p^\prime } \\
\le & C \left( \rho^n \int_{1-R/\rho}^{1+R/\rho}   \left( 1 -  \frac{1+(u )^2-(R / \rho )^2}{2u}  \right)^{\frac{(n-1)p^\prime }{2}} u^{n-1}  \; du \right)^{1/p^\prime } \\
= & C \left( \rho^n \int_{1-R/\rho}^{1+R/\rho}   \left(  \frac{ (R / \rho )^2-(1-u)^2}{2u}  \right)^{\frac{(n-1)p^\prime }{2}} u^{n-1}  \; du \right)^{1/p^\prime } \\
\le & C \left( \rho^n \int_{1-R/\rho}^{1+R/\rho}   \left(  \frac{1}{u} \left( R/\rho \right)^2   \right)^{\frac{(n-1)p^\prime }{2}} u^{n-1}  \; du \right)^{1/p^\prime } \\
\le & C R^{n-1}\left( \rho^{n-(n-1)p^\prime} \int_{1-R/\rho}^{1+R/\rho}    u^{( n-1)(1-\frac{p^\prime}{2} )}  \; du \right)^{1/p^\prime }
\end{align*}

Since we are assuming $\rho>2R$, we have that $1-\frac{R}{\rho}>\frac12$   and  
$$ A( \rho) \le C  R^{n-1} \left( \rho^{n-(n-1)p^\prime }  \frac{R}{\rho} \right)^{1/p^\prime} \le  C R^{n-1/p } \rho^{-(n-1)/p}
$$
whence, in this case,
$$ |f*\chi_{B(0,R)}(x) | \le C R^{n-1/p } \rho^{-(n-1)/p} \Vert f \Vert_{L^p(\R^n)} $$
It remains to prove the case $\rho < 2R $, where 
$$|f*\chi_{B(0,R)}(x) | \le C  R^{n/p^\prime}  \Vert f \Vert_{L^p(\R^n)} \le C  R^{n-{1}/{p}} \rho^{-(n-1)/p}
\Vert f \Vert_{L^p(\R^n)}$$

\end{proof}

\begin{theorem}[Strauss' inequality for potential spaces]
Assume that $1<p<\infty$ and $1/p<s<n$. Let $u \in H_{rad}^{s,p}(\R^n)$. 
Then $u$ is almost everywhere equal to a continuous function in $\R^n-\{0\}$ that satisfies
\be |u(x)| \leq C |x|^{-(n-1)/p} \| u \|_{H^{s,p}(\R^n)} 
\label{strauss1} \ee
and, moreover,
\be | u(y) | \leq  |y|^{-(n-1)/p} \| u \|_{L^p}^{1-1/(sp)} \| D^s u \|_{L^p}^{1/(sp)} 
\label{strauss2}
\ee
\label{fractional-Strauss-ineq}
\end{theorem}

\begin{proof}
Let $u \in  H^{s,p}(\R^n)$ be a radial function, and $f \in L^p(\R^n)$ radial, such that  $u=G_s*f$. Then, for $a>0$ we have that, by Lemma \ref{radial-conv}

\begin{align*}
G_s*f(x)&= \sum_{k \in \Z } \int_{ \lbrace 2^{k-1}a \le |y| \le 2^k a \rbrace } f(x-y)G_s(y) \; dy \\
&\le  \sum_{k \in \Z } G_s(2^{k-1}a) \int_{ \lbrace   |y| \le 2^k a \rbrace } f(x-y) \; dy \\
&=  \sum_{k \in \Z } G_s(2^{k-1}a)\; f*\chi_{B(0,2^k a)}(x)\\
&\le \sum_{k \in \Z } G_s(2^{k-1}a)\; |f*\chi_{B(0,2^k a)}(x)|\\
& \le C \Vert f \Vert_{L^p(\R^n)}|x|^{-(n-1)/p}
\sum_{k \in \Z } G_s(2^{k-1}a) (2^k a)^{n-1/p } 
\end{align*}
so, letting $r_k=2^{k-1}a$ and $\Delta r_k=r_{k+1}-r_k=2^{k-1}a $, we can write the above sum as
$$ C\sum_{k \in \Z } G_s(r_k) (r_k)^{n-1/p }\Delta r_k \to \int_0^\infty G_s(r)r^{n-1/p }\; dr $$ 
when $a\to 0$, so we obtain
\begin{align*} | G_s*f(x)| \le & C \Vert f \Vert_{L^p(\R^n)}|x|^{-(n-1)/p} \int_0^\infty G_s(r)r^{n-1/p }\; dr \\
= & C \Vert f \Vert_{L^p(\R^n)}|x|^{-(n-1)/p} \int_{\R^n} G_s(x)|x|^{-1/p }\; dx \\
\le & C \Vert f \Vert_{L^p(\R^n)}|x|^{-(n-1)/p}
\end{align*}
where the last inequality holds since we are assuming $s>1/p$.

Therefore, 
$$|u(x)|= | G_s*f(x)| \le  C |x|^{-(n-1)/p}  \Vert f \Vert_{L^p(\R^n)}= C |x|^{-(n-1)/p}\Vert u \Vert_{H^{s,p}(\R^n)}$$
which proves \eqref{strauss1}.

We proceed now to the proof of \eqref{strauss2}. 
Let $u_\lambda(x)= u(\lambda x)$. Then,  using Theorem \ref{characterization}, it is easy to see that 
$u_\lambda \in H^{s,p}(\R^n)$ and, moreover,  we have that 
$$ D^s u_\lambda(x)= \lambda^s D^s u(\lambda x) $$
Hence, applying \eqref{strauss1}, 
\begin{align*}
 | u_\lambda(x) | &  \leq C |x|^{-(n-1)/p} \| u_\lambda \|_{H^{s,p}} \\
& \leq C |x|^{-(n-1)/p}  \left[ \; \| u_\lambda \|_{L^p}^p 
+  \| D^s u_\lambda \|_{L^p}^ p \; \right]^{1/p} \\
&\leq C |x|^{-(n-1)/p}  \left[ \; \lambda^{-n}  \| u \|_{L^p}^p 
+ \lambda^{sp-n} \| D^s u\|_{L^p}^p \; \right]^{1/p} 
\end{align*}

Setting $y=\lambda x$,
\begin{align*}
| u(y) |^p   & \leq  C |y|^{-(n-1)}  \lambda^{n-1} \left[ \; \lambda^{-n}  \| u\|_{L^p}^p 
+ \lambda^{sp-n} \| D^s u \|_{L^p}^p \; \right] \\
& \leq  C |y|^{-(n-1)}  \left[ \; \lambda^{-1}  \| u \|_{L^p}^p 
+ \lambda^{sp-1} \| D^s u \|_{L^p}^p \; \right] 
\end{align*}
Now, we choose $\lambda>0$ such that
$$  \lambda^{-1}  \| u \|_{L^p}^p = \lambda^{sp-1} \| D^s u \|_{L^p}^p $$
i.e.,
$$ \lambda= \left( \frac{\| u \|_{L^p}}{\| D^s u \|_{L^p}} \right)^{1/s} $$
Hence:
$$ | u(y) |^p   \leq  2C |y|^{-(n-1)} \| u \|_{L^p}^{p-1/s} \| D^s u\|_{L^p}^{1/s} $$

$$ | u(y) | \leq  2C |y|^{-(n-1)/p} \| u \|_{L^p}^{1-1/(sp)} \| D^s u \|_{L^p}^{1/(sp)} $$

\end{proof}

\begin{remark}
Inequality \eqref{strauss1} was proved  in \cite{SiSk} in the  more general framework of Triebel-Lizorkin spaces,  using an atomic decomposition adapted to the radial situation, which is much more technically involved. A version of Strauss' inequality for a class of Orlicz-Sobolev spaces is given in \cite{AFS}.
\end{remark}

\section{Embedding theorems with power weights for radial functions}

We begin this section by proving a generalization of Ni's inequality. 

\begin{theorem}[Generalization of Ni's inequality]
Let $1<p<\infty$, $u \in H^s_{rad}(\R^n)$ and $1/p<s<n/p$. Then $u$ is almost everywhere equal to a continuous function in $\R^n-\{0\}$ that satisfies
$$ |u(x)| \leq C |x|^{-\frac{n}{p}+s} \;  \|  u\|_{H^{s,p}(\R^n)} $$
\label{Ni-potential-spaces}
\end{theorem}
\begin{proof}
Let $u \in H^{s,p}_{rad} (\R^n) $, then there exists $f \in L^p_{rad}(\R^n) $ such that $u=G_s*f$, where $G_s$ is as in \eqref{Bessel-potential}.
Using \eqref{fractional-integral-bound} we obtain
$$ \Vert |x|^{ n/p-s} u\Vert_{L^\infty(\R^n)} = \Vert |x|^{n/p-s} (G_s*f) \Vert_{L^\infty(\R^n)}\le  C \Vert |x|^{n/p-s} I^s(|f|) \Vert_{L^\infty(\R^n)}.$$
The above inequality combined with Remark \ref{remark-infty} gives the proof. 
\end{proof}

For $s=1$ and $p=2$ this result coincides with Theorem \ref{Ni-H1} in $\R^n$, and for arbitrary $p$ and $s=1$ gives the result in \cite[Lemma 1]{SWW2}.

Now we can proceed to an immediate extension of the result above, that gives an embedding theorem in the critical case (cf. \cite[Lemma 2]{SWW2} when $s=1$).

\begin{theorem}
Let $1<p<\infty$, $0<s<n/p$,  $c>-n$ be such that   $ (1-sp)c \le (n-1)ps $, and let $p^*_c=\frac{p(n+c)}{n-sp}$. Then
$$ \Vert |x|^{c/p^*_c} u \Vert_{L^{p^*_c}(\R^n)} \le C \; \Vert u\Vert_{H^{s,p}(\R^n)}$$
for any  $u\in H^{s,p}_{rad}(\R^n)$.
\label{embedding-critico}
\end{theorem}

\begin{proof}
Using the same argument as in the proof of Theorem \ref{Ni-potential-spaces} we obtain
$$ \Vert |x|^{ c/p^*_c} u \Vert_{L^{p^*_c}(\R^n)} \le  C \Vert |x|^{ c/p^*_c} I^s(|f|) \Vert_{L^{p^*_c}(\R^n)}.$$
We apply Theorem  \ref{theorem-DDD2} with $q=p^*_c $, $\alpha =0$ and $\beta = \frac{-c}{p^*_c}$. Clearly, since $\alpha =0$, it holds that  $\alpha < n/p^\prime$, and $\beta= \frac{-c}{p^*_c} < \frac{n}{p^*_c} < \frac{n}{q}$ since $c >-n$ and $ q < p^*_c $.
\end{proof}

To prove the continuity of the embeddings in the subcritical case we will make used of the following weighted convolution theorem for radial functions, proved by the authors in  \cite[Theorem 5]{DD1}, which shows that the analogous result of R. Kerman \cite[Theorem 3.1]{K} for arbitrary functions can be improved in the radial case.

\begin{theorem}
\label{teo-radiales}
$$ \| |x|^{-\gamma} (f*g)(x) \|_{L^r(\R^n)} 
\leq C \| |x|^\alpha f \|_{L^p(\R^n)} \| |x|^\beta g \|_{L^q(\R^n)} $$
for $f$ and $g$ radially symmetric, provided
\begin{enumerate}
\item $\frac{1}{r} = \frac{1}{p} + \frac{1}{q}+ \frac{\alpha + \beta + \gamma}{n} -1 , 1<p,q,r<\infty ,  \frac{1}{r} \le \frac{1}{p} + \frac{1}{q},$
\item $\alpha<\frac{n}{p'} , \beta< \frac{n}{q'} , \gamma < \frac{n}{r},$
\item $\alpha + \beta \ge {(n-1)(1 - \frac{1}{p} - \frac{1}{q})} , \beta + \gamma \ge (n-1)(\frac{1}{r}- \frac{1}{q}), \gamma + \alpha \ge (n-1)(\frac{1}{r} - \frac{1}{p})$
\item  $\max\{\alpha, \beta , \gamma\} > 0$ or $\alpha=\beta=\gamma=0$.
\end{enumerate}

\end{theorem}

\begin{theorem}
\label{embedding-theorem}
Let $1<p<\infty$, $0<s<\frac{n}{p}$, $p \leq r \leq p^*_c = \frac{p(n+c)}{n-sp}$.
Then we have a continuous embedding
\be
H^{s,p}_{rad} (\mathbb{R}^n) \subset L^r(\mathbb{R}^n,|x|^c dx)
\label{our-embedding}
\ee
provided that
\be
-sp < c < \frac{(n-1)(r-p)}{p}
\label{cota-del-peso}
\ee
\end{theorem}

The case $s=1$ was proved by W. Rother \cite{Rother}. A different proof for the case $p=2$ was given by  the authors and R. Dur\'an in \cite{DDD1}, where this embedding theorem is used to prove existence of solutions of a weighted Hamiltonian elliptic system.

\begin{proof}
Notice that the case $r=p_c^*$ corresponds to Theorem \ref{Ni-potential-spaces}. 

For the remaining cases, we can write, as before, $u=G_s *f$ with $f\in L^p_{rad}(\R^n)$. Using Theorem \ref{teo-radiales} we then have  
$$
\||x|^{c/r} G_s * f\|_{L^r} \le C \|f\|_{L^p} \||x|^\beta G_s\|_{L^q}
$$
provided $\frac{1}{r}=\frac{1}{p}+\frac{1}{q}+\frac{\beta-c/r}{n}-1$, $\frac{1}{r}\le \frac{1}{p}+\frac{1}{q}$, $\beta<\frac{n}{q'}$, $-\frac{c}{r} <\frac{n}{r}$, $\beta\ge (n-1)(1-\frac{1}{p}-\frac{1}{q})$, $\beta-\frac{c}{r} \ge (n-1)(\frac{1}{r}-\frac{1}{q})$, $-\frac{c}{r}\ge (n-1)(\frac{1}{r}-\frac{1}{p}) $ and either $\max\{\beta, -\frac{c}{r}\}>0$ or $\beta=-\frac{c}{r}=0$. 

Since we are assuming $r < \frac{p(n+c)}{n-sp}$, there exists $\varepsilon>0$ such that $r=\frac{p(n+c)}{(n-sp+\varepsilon p)}$. For this value of $\varepsilon$ and sufficiently large $q<\infty$ to be chosen later  set $\beta=\frac{n}{q'}-s+\varepsilon$. 

This choice of $\beta$ clearly makes the scaling condition hold, so it suffices to check the remaining conditions:
\begin{itemize}
\item $\frac{1}{r}\le \frac{1}{p}+\frac{1}{q}$ follows from the fact that $p\le r$. 
\item $\beta<\frac{n}{q'}$ is equivalent to $\varepsilon< s$. Let us check that this is  the case: since $(\varepsilon -s)rp =pn +pc - rn$,  it suffices to check that the RHS is negative, but this is equivalent to $c<\frac{n(r-p)}{p}$, which is true because $r\ge p$ and $c<\frac{(n-1)(r-p)}{p}$ by hypothesis.
\item $-\frac{c}{r}<\frac{n}{r}$ is immediate from $c>-ps>-n$.
\item $\beta\ge (n-1)(1-\frac{1}{q}-\frac{1}{p})$ is equivalent to $c>\frac{r}{p}+\frac{r}{q}-r-n$. Since we already know that $c>-n$, it suffices to check that $\frac{r}{p}+\frac{r}{q}-r <0$ which holds for sufficiently large $q$ because $p>1$. We pick any $q$ that satisfies this condition.
\item $\beta-\frac{c}{r}\ge(n-1)(\frac{1}{r}-\frac{1}{q})$ is equivalent to $\frac{1}{q}\le n+\frac{1}{r}$ which trivially holds.
\item $-\frac{c}{r} \ge (n-1)(\frac{1}{r}-\frac{1}{p})$ is equivalent to $c<\frac{(n-1)(r-p)}{p}$ which holds by hypothesis.
\item $\max\{\beta , -\frac{c}{r}\}>0 $ obviously holds for $c<0$, so let us assume  that $c>0$. In this case, using the scaling condition we have that $\frac{\beta}{n}>\frac{1}{r}-\frac{1}{p}-\frac{1}{q}+1$ which is clearly positive since we chose $q$ to satisfy $-\frac{1}{p}-\frac{1}{q}+1>0$.
\end{itemize}

To complete the proof, we have to check that
$\||x|^\beta G_s\|_{L^q} < C$ which, using \eqref{asymptotics},  holds provided 
$$
\int_0^2 r^{\beta q} r^{(s-n)q} r^{n-1} \, dr < +\infty
$$
that is, $\beta > \frac{n}{q'}-s$, which is true by our choice of $\beta$.
\end{proof}

\begin{remark}
A more general form of Theorem \ref{embedding-theorem} can be proved in the context of weighted Triebel-Lizorkin spaces. Indeed, for power weights $|x|^c$ with $-n<c<n(p-1)$ we have that $H^{s,p}(\R^n)=F^s_{p,2}(\R^n)$ and $L^p(\R^n, |x|^c \, dx)=F^0_{p,2}(\R^n, |x|^c \, dx)$. One way of proving the theorem in this setting is by means of a weighted Plancherel-Polya-Nikol'skij type inequality, which in turn makes use of a weighted convolution theorem. In the general setting this approach was used by M. Meyries and M. Veraar in \cite{MV}. In the radial setting, the authors proved in \cite{DD1} that a better Plancherel-Polya-Nikol'skij equality holds in the case of radial functions, which in turn gives better weighted embeddings.

More technically involved proofs use atomic decompositions or wavelet decompositions but have the advantage of allowing more general weights. See, e.g., \cite{HS1, HS2, HP}. In the case of radial functions, this is the subject of a forthcoming paper of the authors and N. Saintier \cite{DDS}. 
\end{remark}

\section{Compactness of the embeddings }

As announced in the introduction, we begin this section by proving the following theorem due to  P.L. Lions \cite{Lions}, that we will need later to prove the compactness in the weighted case. Our proof is  different from the original one of P. L. Lions in \cite{Lions} and avoids the use of the complex method of interpolation.

\begin{theorem}
For $1<p<q<p^*$, we have a compact embedding
$$ H^{s,p}_{rad}(\R^n) \subset L^q(\R^n) $$
for any $s>0$.
\end{theorem}

\begin{proof}
Let $(u_n)$ be a bounded sequence in $H^{s,p}_{rad}(\R^n)$,
\be \| u_n \|_{H^{s,p}(\R^n)} \leq M \label{hip} \ee
i.e. $u_n= G_s * f_n$ with radial $f_n$ such that
$$ \| f_n \|_{L^p} \leq M $$

By the Kolmogorov-Riesz theorem (see \cite{HH}) we need to check three conditions:
\begin{enumerate}
\item[i)] $(u_n)$ is bounded in $L^q(\R^n)$. 
This is clear by \eqref{hip} and the continuity of the Sobolev embedding \eqref{Sobolev-embedding}.

\item[ii)] (equicontinuity condition) Let $\tau_h u(x)= u(x+h)$.
Given $\varepsilon>0$, there exists $\delta=\delta(\varepsilon)>0$ such that 
$$ \| \tau_h u_n - u_n \|_{L^q} < \varepsilon $$
if $|h|<\delta$.

This is easy since:
$$ 
\| \tau_h u_n - u_n \|_{L^q} = \| \tau_h(G_s *  f_n) - G_s*f_n \|_{L^q}
$$
$$ = \| (\tau_h G_s - G_s)* f_n \|_{L^q} 
\leq \| \tau_h G_s - G_s \|_{L^r} \| f_n \|_{L^p}  $$
if $$ \frac{1}{q}+1= \frac{1}{r}+\frac{1}{p} $$
by Young's inequality, but, since $G_s \in L^q$, 
$$ \| \tau_h G_s - G_s \|_{L^r} < \frac{\varepsilon}{M} $$
if $|h|< \delta$. We conclude that
$$ \| \tau_h u_n - u_n \|_{L^q} \leq \varepsilon $$
if $|h|<\delta$.

\item[iii)] (tightness condition) Given $\varepsilon>0$, there exists $R=R(\varepsilon)>0$ such
that
$$ \int_{|x|>R} |u_n|^q \; dx < \varepsilon \; \hbox{ for all} \; n \in \N $$

In order to check this condition, we observe that if
$s>\frac{1}{p}-\frac{1}{q}$, by Theorem \ref{embedding-theorem}, we have that
$$ R^{-\gamma q} \int_{|x|>R} |u_n|^q \leq \int_{|x|>R} |x|^{-\gamma q}
|u_n|^q 
\leq C \| u_n \|_{H^{s,p}(\R^n)}^{q} $$
where
$$ \gamma= (n-1)\left(\frac{1}{q} - \frac{1}{p}\right) <0  $$
Hence
\be  
\int_{|x|>R} |u_n|^q  \leq \frac{C} {R^{-\gamma q}} M^q < \varepsilon 
\label{tail-bound}
\ee
if we choose $R=R(\varepsilon)$ large enough.

We observe that $\delta= -\gamma q$ is exactly the exponent in Lions' paper
\cite{Lions} (second equation on page 321), but he proves \eqref{tail-bound} using an
interpolation argument.

The case $s \leq \frac{1}{p}-\frac{1}{q}$ cannot happen under the theorem hypotheses since
$$ q < p^* \Rightarrow s > n \left(\frac{1}{p}-\frac{1}{q} \right) \geq \frac{1}{p}-\frac{1}{q} $$

\end{enumerate}
\end{proof}

\begin{theorem}
The embedding
\begin{equation*}
H^{s,p}_{rad} (\mathbb{R}^n) \subset L^r(\mathbb{R}^n,|x|^c dx)
\end{equation*}
of Theorem \ref{embedding-theorem} is compact, provided that  $1<p<\infty$, $0<s<\frac{n}{p}$, $p < r < p^*_c = \frac{p(n+c)}{n-sp}$ and $-sp < c < \frac{(n-1)(r-p)}{p}$.
\end{theorem}
\begin{proof}
 It is enough to
show that if $u_n \to 0 $ weakly in $H_{rad}^{s,p}(\R^n)$, then $u_n \to 0$
strongly in $L^r(\R^n,|x|^c \; dx)$. Since 
$$p<r<\frac{p(n+c)}{n-sp}$$ by hypothesis, it is possible
to choose $q$ and $\tilde{r}$ so that $p<q<r<\tilde{r}<\frac{p(n+c)}{n-sp}$.
We write $r = \theta q + (1-\theta) \tilde{r}$ with $\theta \in
(0,1)$ and, using H\"older's inequality, we have that

\be \int_{\R^n} |x|^{c} |u_n|^r \, dx \le \left( \int_{\R^n}
|u_n|^q \, dx \right)^{\theta} \left( \int_{\R^n} |x|^{\tilde{c}}
|u_n|^{\tilde{r}} \, dx \right)^{1-\theta}
\label{ineq-interpolacion} \ee where
$\tilde{c}=\frac{c}{1-\theta}$. By choosing $q$ close enough to
$p$ (hence making $\theta$ small), we can fulfill the conditions
$$ \tilde{r} <\frac{p(n+\tilde{c})}{n-ps},\quad -ps < \tilde{c} < \frac{(n-1)(\tilde{r}-p)}{p}. $$
Therefore, by the imbedding that we have already established:
$$ \left( \int_{\R^n} |x|^{\tilde{c}} |u_n|^{\tilde{r}} \, dx  \right)^{1/\tilde{r}}
\leq C \| u_n \|_{H^{s,p}} \leq C$$

Since the imbedding $H_{rad}^{s,p}(\R^n) \subset L^{q}(\R^n)$ is compact by Lions' theorem, we
have that $u_n \to 0$ in $L^{q}(\R^n)$. From (\ref{ineq-interpolacion}) we conclude
that $u_n \to 0$ strongly in $L^r(\R^n,|x|^c\; dx)$, which shows that the imbedding in our
theorem is also compact. This concludes the proof.
\end{proof}

\section{Ni's inequallity for potential spaces in a ball}

We recall that the original result of Ni in \cite{Ni} was for the case of
the ball. Hence, it is natural to ask if our extension of Ni's inequality also
holds for potential spaces of radial functions in a ball. In this section we
briefly discuss this extension.

Let $B=B(0,R)=\{ x \in \R^n : |x|<R\}$ be a ball. We denote by $\Delta_B$ the
Laplacian operator with Dirichlet conditions. Let $(\lambda_k)_{k \in \N}$ 
be the Dirichlet eigenvalues, with the corresponding orthogonal 
basis of $L^2(B)$ of eigenfunctions $(\lambda_k)_{k \in \N}$. Then the negative powers
of $-\Delta_B$ can be defined in $L^2(B)$ by 
$$ (-\Delta)^{-s/2} f(x)= \sum_{k} \lambda_{k}^{-s/2} 
\langle f, \varphi k \rangle  \varphi_k(x) $$

Let 
$$ H^t_{B}(x,y)= \sum_{k} e^{-\lambda_k t} \varphi_k(x) \varphi_k(y) $$
be the heat kernel for $B$. Then, from the well-known  formula,
$$ (-\Delta)^{-s/2} f(x)= \frac{1}{\Gamma(s/2)} \int_0^\infty t^{s/2-1} \; 
e^{t \Delta} \; f(x) \;  dt $$
we see that $(-\Delta)^{-s/2}$ has an integral representation
$$ (-\Delta)^{-s/2} f(x) = \int_{B} K^s(x,y) \; f(y) \; dy $$
where the kernel $K^s$ is given by
$$ K^s(x,y) \frac{1}{\Gamma(s/2)} \int_0^\infty t^{s/2-1} \; 
H^t_B(x,y) \; f(x) \;  dt $$
and this formula makes sense for $f \in L^p(B)$. Hence, we may define the 
potential spaces for the ball

$$ H^{s,p}_0(B)= \{ u:  u = (-\Delta_B)^{-s/2} f \; \hbox{with} \; f  
\in L^p(B)  \} $$

(We consider the operator $-\Delta_B$ and not $I-\Delta_B$ in the definition of
these spaces since the first eigenvalue $\lambda_1$ of the Laplacian in $B$
is strictly positive). For $p=2$ these spaces are useful in the study of elliptic systems by variational methods (see, e.g.,  \cite{dFF, dFPR}).

The parabolic maximum principle implies that $H^t$ is bounded by the heat kernel of
the whole space $\R^n$:
$$ 0 \leq H^t_{B}(x,y) \leq H^t_{R^n}(x,y) =\frac{1}{(4\pi t)^{n/2}} e^{-|x-y|^2/4t} $$
Hence, we deduce the bound
$$ 0 \leq K^s(x,y) \leq \frac{C(n,s)}{|x-y|^{n-s} } \quad 0 <s < n $$

It follows that we have the pointwise estimate
$$ |(-\Delta_B)^{-s/2} f(x)| \leq C(n,s) \; I^s(|\tilde{f}|)(x) $$
where we denote by $\tilde{f}$ the extension of $f$ by zero outside $B$.
Using Theorem \ref{theorem-DDD2} (as in the proof of theorems
\ref{Ni-potential-spaces} and \ref{embedding-critico}),  we immediately get

\begin{theorem}[Generalization of Ni's inequality for the ball]
Let $1<p<\infty$, $u \in H^{s,p}_{0,rad}(B)$ and $1/p<s<n/p$. Then $u$ is almost everywhere equal to a continuous function in $B-\{0\}$ that satisfies
$$ |u(x)| \leq C |x|^{-\frac{n}{p}+s} \;  \|  u \|_{H^{s,p}_0(B)} $$
\label{Ni-ball}
\end{theorem}

\begin{theorem}
Let $1<p<\infty$, $0<s<n/p$,  $c>-n$ be such that   $ (1-sp)c \le (n-1)ps $, and let $p^*_c=\frac{p(n+c)}{n-sp}$ . Then
$$ \Vert |x|^{c/p^*_c} u\Vert_{L^{p^*_c}(\R^n)} \le C \; \Vert u
\Vert_{H^{s,p}_0(B)}$$
for any radial function $u \in H^{s,p}_{0,rad}(B)$.
\end{theorem}

Related inequalities for Sobolev spaces of integral order in a ball are proved in \cite{DMM}. In particular, the exponent in Theorem \ref{Ni-ball} coincides for $p=2$ and integral $s$ with that of \cite[Theorem 1.1(2)]{DMM}.

\end{document}